\documentclass{amsart}
\usepackage{amssymb}

\theoremstyle{plain}
\newtheorem{theorem}{Theorem}[section]
\newtheorem{corollary}[theorem]{Corollary}
\newtheorem{definition}[theorem]{Definition}
\newtheorem{example}[theorem]{Example}
\newtheorem{proposition}[theorem]{Proposition}
\newtheorem{lemma}[theorem]{Lemma}
\numberwithin{equation}{section}

\begin{document}

\title[Connections between partitions]{Connections between colored restricted 
$b$-ary and ordinary partitions}

\author{Karl Dilcher}
\address{Department of Mathematics and Statistics\\
         Dalhousie University\\
         Halifax, Nova Scotia, B3H 4R2, Canada}
\email{dilcher@mathstat.dal.ca}
\author{Larry Ericksen}
\address{1212 Forest Drive, Millville, NJ 08332-2512, USA}
\email{LE22@cornell.edu}

\date{}

\setcounter{equation}{0}

\begin{abstract}
We establish various connections between classes of colored and bounded 
ordinary partitions on one hand, and colored but not necessarily bounded 
binary and $b$-ary partitions on the other hand. Many of these results are 
based on special recurrence relations, some of which are new. We also obtain
several classes of identities for sequences of colored $b$-ary partitions,
and there are a few results concerning compositions.
\end{abstract}

\maketitle

\section{Introduction}\label{sec:1}

Integer partitions can be restricted in a variety of ways, for instance by
limiting the number or the sizes of their parts. This can also be done for
binary partitions, that is, for writing an integer $n\geq 1$ as a sum of
powers of 2, repetitions permitted. Colored ordinary and binary partitions,
both restricted and unrestricted, have also been studied in recent years.
Finally, binary partitions can be extended in an obvious way to $b$-ary
partitions for a fixed integer base $b\geq 2$. A brief historical overview
of these concepts can be found in \cite{DE10}.

So far, ordinary partitions and binary (or $b$-ary) partitions and their
respective restricted versions have mostly been considered as rather different
objects. The only connection between the two of which we are aware is the 
following result.
 
\begin{proposition}[Hirschhorn and Sellers \cite{HS}]\label{prop:1.1}
Given an integer $n\geq 1$, the number of binary partitions of $n$ is
equal to the number of partitions
\[
n=p_k+p_{k-1}+\cdots+p_1,\qquad p_k\geq p_{k-1}\geq\cdots\geq p_1\geq 1,
\]
satisfying
\[
p_{i+1}\geq p_i+p_{i-1}+\cdots+p_1\quad\hbox{for all}\quad 1\leq i\leq k-1.
\]
\end{proposition}

Hirschhorn and Sellers actually proved the general case for $b$-ary partitions,
and Sloane and Sellers \cite{SS} called such partitions {\it non-squashing}.
See also \cite{AS} and \cite[A018819]{OEIS}.

It is the purpose of this paper to identify and study some further close
connections between the two basic classes of partitions, namely ordinary and
$b$-ary partitions and their respective colored and bounded analogues.
We begin with the interesting infinite product
\begin{equation}\label{1.1}
\prod_{j=0}^\infty\left(1+q^{2^j}\right)=\frac{1}{1-q}\qquad (|q|<1),
\end{equation}
which can be found, for instance, in the tables \cite[(89.30.2)]{Ha} or 
\cite[(6.2.3.1)]{PrE}. If we write the right-hand side as $1+q+q^2+\cdots$ and
expand the left-hand side, we see that \eqref{1.1} is equivalent to the 
uniqueness of binary expansions. This fact can also be seen as a result on
binary partitions, and \eqref{1.1} turns out to be the main tool in the first
part of this paper. This will be done in Section~\ref{sec:2}, along with a 
$b$-ary analogue of \eqref{1.1}.

However, the main purpose of Section~\ref{sec:2} is to prove some connections
between restricted ordinary partitions and colored binary and $b$-ary 
partitions. In Section~\ref{sec:3} we recall some known generating functions
and recurrence relations for sequences of restricted $b$-ary partitions and
their polynomial analogues. These recurrences are then used in 
Section~\ref{sec:4} to establish connections between various types of 
restricted colored $b$-ary and restricted colored ordinary partitions.

The final two sections deal with unrestricted colored $b$-ary partitions. In
Section~\ref{sec:5}, we derive some relevant new recurrence relations, which
are then used in Section~\ref{sec:6} to obtain various analogues to results in
Section~\ref{sec:4}. Finally, we derive several classes of identities for 
sequences of colored $b$-ary partitions.

\section{Colored $b$-ary partition functions}\label{sec:2}

In dealing with {\it colored} ordinary or $b$-ary partitions, we assign 
different ``colors" by attaching numbered subscripts to the parts of a
partition. Using the exposition in \cite{DE10}, we give the following 
definition and example.

\begin{definition}\label{def:2.1}
Let $b\geq 2$ and $\rho\geq 1$ be integers, and let
$\Lambda=(\lambda_1,\ldots,\lambda_\rho)$ be a finite integer sequence with
$1\leq\lambda_1\leq\cdots\leq\lambda_\rho$. Then a $\Lambda$-restricted 
$\rho$-color $b$-ary partition of an integer $n\geq 1$ is a set
of powers $b^j$ whose sum is $n$, and where for a given integer 
$j\geq 0$, up to $\lambda_r$ parts $b^j$ can be assigned the $r$th color,
for $1\leq r\leq\rho$.
\end{definition}

In other words, each part $b^j$ can be assigned the $r$th color at most
$\lambda_r$ times, for $1\leq r\leq\rho$. 
If $C_b^\Lambda(n)$ denotes the number of partitions of $n$ as defined in 
Definition~\ref{def:2.1}, then the generating function is
\begin{equation}\label{2.1}
\sum_{n=0}^\infty C_b^\Lambda(n)q^n
= \prod_{j=0}^\infty\bigg(1+q^{b^j}+\cdots+q^{\lambda_1\cdot b^j}\bigg)\cdots
\bigg(1+q^{b^j}+\cdots+q^{\lambda_\rho\cdot b^j}\bigg). 
\end{equation}

We now illustrate Definition~\ref{def:2.1} with a concrete example.

\begin{example}\label{ex:2.2}
{\rm Let $b=2$ and $\Lambda=(2,3)$, and thus $\rho=2$. Then the 
$\Lambda$-restricted 2-color binary partitions of $n=4$ are}

$4_1,\;4_2,\;2_1+2_1,\;2_1+2_2,\;2_2+2_2,\;2_1+1_1+1_1,\;2_1+1_1+1_2,\;
2_1+1_2+1_2,$

$2_2+1_1+1_1,\;2_2+1_1+1_2,\; 2_2+1_2+1_2,\;1_1+1_1+1_2+1_2,\;
1_1+1_2+1_2+1_2,$

\noindent
{\rm where a subscript indicates the color of a part (e.g., $1=red$, $2=blue$),
and by convention we 
write each partition in non-increasing order of its parts.
Counting the partitions in this example, we see that $C_2^\Lambda(4)=13$. 
On the other hand, by \eqref{2.1} we have}
\begin{align}
\sum_{n=0}^\infty C_2^\Lambda(n)q^n
&= \prod_{j=0}^\infty\bigg(1+q^{2^j}+q^{2\cdot 2^j}\bigg)
\bigg(1+q^{2^j}+q^{2\cdot 2^j}+q^{3\cdot 2^j}\bigg)\label{2.2} \\
&= 1+2q+5q^2+7q^3+13q^4+17q^5+26q^6+\ldots,\nonumber
\end{align}
{which confirms that $C_2^\Lambda(4)=13$.}
\end{example}

We are now ready to state and prove the main results of this section.

\begin{proposition}\label{prop:2.4}
Given positive integers $k$ and $n$, the following are identical:
\begin{enumerate}
\item[(a)] The number of partitions of $n$ with parts not exceeding $k$.
\item[(b)] The number of $k$-color binary partitions in which a part with
color $i$ has multiplicity $i$.
\end{enumerate}
\end{proposition}

To supplement this result, we note that by a well-known property of restricted
partitions, (a) is the same as
\begin{enumerate}
\item[(a$'$)] The number of partitions of $n$ with at most $k$ parts.
\end{enumerate}

\begin{example}\label{ex:2.5}
{\rm Let $k=3$ and $n=6$. Then the restricted ordinary partitions satisfying
(a) and (a$'$) are respectively
\begin{gather*}
3+3,\;3+2+1,\;3+1+1+1,\;2+2+2,\;2+2+1+1,\\
2+1+1+1+1,\;1+1+1+1+1+1,
\end{gather*}
and
\[
6,\;5+1,\;4+2,\;4+1+1,\;3+3,\;3+2+1,\;2+2+2,
\]
and we see that in both cases there are 7 partitions.

The 3-color binary partitions satisfying (b) are as follows, where we put
equal parts in parentheses for easier verification:
\begin{gather*}
4_1+2_1,\; 4_1+(1_2+1_2),\; 2_1+(2_2+2_2),\; (2_3+2_3+2_3),\; 
(2_2+2_2)+(1_2+1_2),\\
 2_1+1_1+(1_3+1_3+1_3),\; 1_1+(1_2+1_2)+(1_3+1_3+1_3).
\end{gather*}
There are again seven such partitions, in agreement with 
Proposition~\ref{prop:2.4}.}
\end{example}

Proposition~\ref{prop:2.4} is a special case of a more general result that 
deals with $k$-color $b$-ary partitions for arbitrary bases $b\geq 2$. For the
proof we need the following infinite product.

\begin{lemma}\label{lem:2.6}
For integers $b\geq 2$ and $k\geq 0$ and for complex $q$ with $|q|<1$, we have
\begin{equation}\label{2.3}
\prod_{j=0}^k\left(1+q^{b^j}+q^{2b^j}+\cdots+q^{(b-1)b^j}\right)
=\frac{1-q^{b^{k+1}}}{1-q},
\end{equation}
and consequently
\begin{equation}\label{2.4}
\prod_{j=0}^\infty\left(1+q^{b^j}+q^{2b^j}+\cdots+q^{(b-1)b^j}\right)
=\frac{1}{1-q}.
\end{equation}
\end{lemma}

\begin{proof}
We multiply the left-hand side of \eqref{2.3} by $1-q$. Then we get
consecutively
\begin{align*}
(1-q)\left(1+q+q^2+\cdots+q^{b-1}\right) &= 1-q^b,\\
\left(1-q^b\right)\left(1+q^b+q^{2b}+\cdots+q^{(b-1)b}\right) &= 1-q^{b^2},
\end{align*}
and so on, until
\[
\left(1-q^{b^k}\right)\left(1+q^{b^k}+q^{2b^k}+\cdots+q^{(b-1)b^k}\right)
=1-q^{b^{k+1}}.
\]
This proves \eqref{2.3} and by taking the limit as $k\to\infty$, we 
immediately get \eqref{2.4}.

An alternative proof of \eqref{2.4} is similar to the remark following 
\eqref{1.1}: If we expand the left-hand side of \eqref{2.4}, then by the
uniqueness of the base-$b$ expansion for any $b\geq 2$, the coefficients of
all $q^n$, for $n\geq 1$, will be 1. This gives the right-hand side of 
\eqref{2.4}.
\end{proof}

The identity \eqref{2.4} was earlier used in \cite[Cor.~8.1]{DE10}. Also, it
is clear that \eqref{1.1} is the special case $b=2$ of \eqref{2.4}.

We are now ready to state and prove the following extension of 
Proposition~\ref{prop:2.4}.

\begin{proposition}\label{prop:2.7}
Given positive integers $k$ and $n$, the following are identical:
\begin{enumerate}
\item[(a)] The number of partitions of $n$ with parts not exceeding $k$.
\item[(b)] For any integer $b\geq 2$, the number of $k$-color $b$-ary 
partitions in which a part with color $i$ has multiplicity in the set
$\{i, 2i,\ldots, (b-1)i\}$.
\end{enumerate}
\end{proposition}

\begin{proof}
Let $p_k(n)$ denote the number of partitions of $n$ with parts not exceeding
$k$. It is well known, and can be easily seen, that for fixed $k\geq 1$ the
generating function of the sequence $\left(p_k(n)\right)_{n\geq 0}$ is
\begin{equation}\label{2.5}
\sum_{n=0}^\infty p_k(n)q^n = \prod_{i=1}^k\frac{1}{1-q^i}.
\end{equation}
Now we replace $q$ by $q^i$ in \eqref{2.4} and substitute it into \eqref{2.5},
obtaining
\begin{equation}\label{2.6}
\sum_{n=0}^\infty p_k(n)q^n = \prod_{j=0}^\infty\prod_{i=1}^k
\left(1+q^{ib^j}+q^{2ib^j}+\cdots+q^{(b-1)ib^j}\right).
\end{equation}
Now the right-hand side is seen to be the generating function of $k$-color 
$b$-ary partitions, where the parts of color $i$ have multiplicities
$i, 2i,\ldots$, or $(b-1)i$. This was to be shown.
\end{proof}

We conclude this section with an example.

\begin{example}\label{ex:2.8}
{\rm As in Example~\ref{ex:2.5}, we let $k=3$ and $n=6$, but here we take $b=3$.
The unrestricted ternary partitions of 6 are $3+3$, $3+1+1+1$, and $1+\cdots+1$.
The corresponding 3-color ternary partitions of 6, restricted according to (b)
in Proposition~\ref{prop:2.7}, are then
\begin{gather*}
(3_1+3_1),\; (3_2+3_2),\; 3_1+1_1+(1_2+1_2),\; 3_1+(1_3+1_3+1_3),\\
1_1+(1_2+1_2)+(1_3+1_3+1_3),\; (1_1+1_1)+(1_2+1_2+1_2+1_2),\; (1_3+\cdots+1_3).
\end{gather*}
Once again we see that there are 7 such partitions, this time in agreement 
with Proposition~\ref{prop:2.7}}.
\end{example}

\section{Polynomial analogues and recurrence relations}\label{sec:3}


In \cite{DE10} we defined polynomial analogues of the sequence
$C_b^\Lambda(n)$, with the aim of characterizing all the relevant partitions.
The purpose of the present paper is different, and therefore we only state
the following special case of Definition~2.4 in~\cite{DE10}. This, along with
Proposition~\ref{prop:3.2}, will be essential for the following sections.

\begin{definition}\label{def:3.1}
Let $b\geq 2$ and $\rho\geq 1$ be integers, and let
$\Lambda=(\lambda_1,\ldots,\lambda_\rho)$ be a finite integer sequence with
$1\leq\lambda_1\leq\cdots\leq\lambda_\rho$. Furthermore, let
$\lambda:=\lambda_1+\cdots+\lambda_\rho$, and let
$Z=(z_{1,1},\ldots,z_{1,\lambda_1};\ldots;z_{\rho,1},\ldots,z_{\rho,\lambda_\rho})$
be a $\lambda$-tuple of variables. Then we define the sequence of
$\lambda$-variable polynomials $C_b^\Lambda(n;Z)$ by the generating function
\begin{align}
\sum_{n=0}^\infty C_b^\Lambda(n;Z)q^n
=\prod_{j=0}^\infty&\left(1+z_{1,1}q^{b^j}+\dots
+z_{1,\lambda_1}q^{\lambda_1\cdot b^j}\right)\cdots\label{3.1}\\
&\qquad\left(1+z_{\rho,1}q^{b^j}+\dots
+z_{\rho,\lambda_\rho}q^{\lambda_\rho\cdot b^j}\right).\nonumber
\end{align}
\end{definition}

The polynomials $C_b^\Lambda(n;Z)$ satisfy recurrence relations, which we state
here as a special case of Theorem~4.3 in \cite{DE10}.

\begin{proposition}\label{prop:3.2}
Let $b\geq 2$ and $\rho\geq 1$ be integers, and let the finite sequences
$\Lambda$ and $Z$ be as in Definition~\ref{def:3.1}. 
Then $C_b^\Lambda(0;Z)=1$, and for integers $n\geq 0$ we have
\begin{equation}\label{3.2}
C_b^\Lambda(bn+j;Z) 
= \sum_{k=0}^{\lfloor\lambda/b\rfloor}Y_{bk+j}(Z)C_b^\Lambda(n-k;Z)
\qquad (j=0, 1,\ldots b-1),
\end{equation}
with the convention that $C_b^\Lambda(m;Z)=0$ if $m<0$. The 
coefficients $Y_\nu(Z)$, $0\leq\nu\leq\lambda$, are given by 
\begin{equation}\label{3.3}
Y_\nu(Z) = \sum z_{1,i_1}z_{2,i_2}\cdots z_{\rho,i_\rho},\qquad
z_{1,0}=z_{2,0}=\cdots=z_{\rho,0}=1,
\end{equation}
with the sum taken over all $i_1,i_2,\ldots,i_\rho$ with $i_1+\cdots+i_\rho=\nu$
and $0\leq i_\ell\leq\lambda_\ell$ for $\ell=1,\ldots,\rho$. 
In particular, from \eqref{3.3} we get
\[
Y_0(Z)=1,\quad Y_1(Z)=z_{1,1}+z_{2,1}+\cdots+z_{\rho,1},\quad
Y_\lambda(Z) = z_{1,\lambda_1}z_{2,\lambda_2}\cdots z_{\rho,\lambda_\rho},
\]
and $Y_\mu(Z)=0$ when $\mu>\lambda$.
\end{proposition}

\begin{example}\label{ex:3.3}
{\rm As in Examples~\ref{ex:2.2}, we let $b=2$ and $\Lambda=(2,3)$. To avoid
double subscripts, we let $Z=(y_1,y_2;z_1,z_2,z_3)$, and for the sake of 
simplicity we delete the subscripts and superscripts of $C$. Since $\lambda=5$
and $\lfloor\lambda/b\rfloor=2$, Proposition~\ref{prop:3.2} gives}
\begin{align}
C(2n;Z) &= Y_0(Z)C(n;Z) + Y_2(Z)C(n-1;Z) +Y_4(Z)C(n-2;Z),\label{3.4}\\
C(2n+1;Z) &=  Y_1(Z)C(n;Z) + Y_3(Z)C(n-1;Z) + Y_5(Z)C(n-2;Z),\label{3.5}
\end{align}
{\rm where} 
\begin{gather*}
Y_0(Z)=1,\quad Y_1(Z)=y_1+z_1,\quad Y_2(Z)=z_2+y_1z_1+y_2,\\
Y_3(Z)=z_3+y_1z_2+y_2z_1,\quad Y_4(Z)=y_1z_3+y_2z_2, \quad Y_5(Z)=y_2z_3.
\end{gather*}
{\rm With $Z=(1,1;1,1,1)$, we have the situation of Example~\ref{ex:2.2}, and
suppressing again the subscripts and superscripts of $C$, we get from 
\eqref{3.4} and \eqref{3.5},}
\begin{align}
C(2n) &= C(n) + 3C(n-1) +2C(n-2),\label{3.6}\\
C(2n+1) &=  2C(n) + 3C(n-1) + C(n-2).\label{3.7}
\end{align}
{\rm We can see that the coefficients in the second line of \eqref{2.2} 
satisfy these last two recurrences.}
\end{example}

Determining the coefficient polynomials $Y_\nu(Z)$, defined by \eqref{3.3}, is
often difficult and awkward. In some cases, the following generating function
can be helpful.

\begin{proposition}\label{prop:3.4}
The finite sequence of polynomials $Y_\nu(Z)$, defined by \eqref{3.3}, has the
generating function
\begin{equation}\label{3.8}
\sum_{\nu=0}^\lambda Y_\nu(Z)q^\nu = \prod_{i=1}^\rho
\left(1+z_{i,1}q+z_{i,2}q^2+\dots +z_{i,\lambda_i}q^{\lambda_i}\right).
\end{equation}
\end{proposition}

\begin{proof}
We expand the right-hand side of \eqref{3.8} and collect the coefficients of
$q^\nu$, $0\leq\nu\leq\lambda$. Then we see that for each $\nu$ we get the sum
in \eqref{3.3}.
\end{proof}

We have the following formal relationship between \eqref{3.8} and the 
generating function \eqref{3.1}. If we write \eqref{3.1} in the form
\begin{equation}\label{3.9}
\sum_{n=0}^\infty C_b^\Lambda(n;Z)q^n =\prod_{j=0}^\infty\prod_{i=1}^\rho
\left(1+z_{i,1}q^{b^j}+z_{i,2}q^{2\cdot b^j}+\dots
+z_{i,\lambda_i}q^{\lambda_i\cdot b^j}\right),
\end{equation}
then the right-hand side of \eqref{3.8} is the inner product of \eqref{3.9}
with $b=1$.

To obtain a first consequence of Proposition~\ref{prop:3.4}, we set 
$Z=(1,\ldots,1)$ and recall that $\lambda=\lambda_1+\dots+\lambda_\rho$. For
simplicity, we set $Y_n:=Y_n(1,\ldots,1)$.

\begin{corollary}\label{cor:3.5}
We have the generating function
\begin{equation}\label{3.10}
\sum_{n=0}^\lambda Y_nq^n = \prod_{i=1}^\rho
\left(1+q+q^2+\dots+q^{\lambda_i}\right),
\end{equation}
and consequently
\begin{equation}\label{3.11}
\sum_{n=0}^\lambda Y_n = \prod_{i=1}^\rho\left(1+\lambda_i\right),
\end{equation}
Furthermore, we have
\begin{equation}\label{3.12}
Y_n = Y_{\lambda-n}\quad\hbox{for all}\;\;0\leq n\leq\lambda.
\end{equation}
\end{corollary}

\begin{proof}
The identity \eqref{3.10} is a special case of \eqref{3.8}, and \eqref{3.11}
follows from \eqref{3.10} by setting $q=1$. This is allowed since sum and 
product are finite here. Finally, \eqref{3.12} reflects the fact that the
polynomial on the left is self-reciprocal which, in turn, can be seen by
replacing $q$ by $q^{-1}$ in \eqref{3.10} and multiplying both sides 
by~$q^\lambda$.
\end{proof}

\section{Connections through recurrences}\label{sec:4}

\subsection{All parameters $z_{i,j}=1$}\label{sec:4.1}

For a combinatorial interpretation of \eqref{3.10}, we note that upon 
expanding the product on the right of \eqref{3.10}, we can assign the 
color $i$ to the factor
belonging to $i$, for $1\leq i\leq\rho$. We then get the following 
characterization of $Y_n$ in terms of restricted colored ordinary partitions.

\begin{lemma}\label{lem:3.6}
For $1\leq n\leq\lambda$, let $Y_n$ be as in \eqref{3.10}. Then $Y_n$ counts 
the number of $\rho$-color ordinary partitions of $n$ which have at most
one part of each color and a part of color $i$ is bounded by $\lambda_i$.
\end{lemma} 

\begin{example}\label{ex:3.7}
{\rm As in Examples~\ref{ex:2.2} and~\ref{ex:3.3}, we let $\Lambda=(2,3)$. Then
the partitions described in Lemma~\ref{lem:3.6} are as follows:}
\begin{center}
\begin{tabular}{l|l}
$n=1$:\quad $1_1,\;1_2$ & $n=4$:\quad $3_2+1_1,\;2_1+2_2$ \\
$n=2$:\quad $2_1,\;2_2,\;1_1+1_2$ & $n=5$:\quad $3_2+2_1$ \\
$n=3$:\quad $3_2,\;2_1+1_2,\;2_2+1_1$ &
\end{tabular}
\end{center}
{\rm The numbers $Y_n$ of these are in agreement with the corresponding
polynomial~\eqref{3.10}:
\[
\left(1+q+q^2\right)\left(1+q+q^2+q^3\right)
= 1+2q+3q^2+3q^3+2q^4+q^5,
\]
with the coefficients in \eqref{3.6} and \eqref{3.7}.}
\end{example}

For a more general result, we only need to put together the identities 
\eqref{2.1}, \eqref{3.2}, and Lemma~\ref{lem:3.6}.

\begin{proposition}\label{prop:3.8}
Let $b\geq 2$, $\rho\geq 1$, and $\Lambda=(\lambda_1,\ldots,\lambda_\rho)$.
Then the sequence $C_b^\Lambda(n)$ of $\Lambda$-restricted $\rho$-color 
$b$-ary partitions satisfies the recurrence relation
\begin{equation}\label{3.13}
C_b^\Lambda(bn+j)=\sum_{k=0}^{\lfloor\lambda/b\rfloor}Y_{bk+j}C_b^\Lambda(n-k)
\qquad (j=0, 1,\ldots b-1),
\end{equation}
where $Y_0=1$ and the coefficients $Y_n$, $n\geq 1$, count the number of 
$\rho$-color ordinary partitions of $n$ which have at most
one part of each color and a part of color $i$ is bounded by $\lambda_i$.
\end{proposition}

Next, we consider the special case where $\lambda_1=\dots=\lambda_\rho$, that 
is, the colored parts are uniformly bounded. In this case we can use
compositions as an alternative to the colored partitions of Lemma~\ref{lem:3.6}.
We recall that an integer composition is a partition in which the order of the
parts matters. We then have the following variant of Lemma~\ref{lem:3.6}.

\begin{lemma}\label{lem:3.9} 
Let $\lambda_1=\dots=\lambda_\rho$ in \eqref{3.10}. Then the following hold:
\begin{enumerate}
\item[(a)] $Y_n$ counts the number of compositions of $n+\rho$ into exactly 
$\rho$ parts between $1$ and $\lambda_1+1$.
\item[(b)] When $\lambda_1=1$, then $Y_n=\binom{\rho}{n}$, $n=0,1,\ldots,\rho$.
\end{enumerate}
\end{lemma} 

\begin{proof}
(a) This is a well-known fact; see, e.g., \cite{Eg}, \cite{HO}, or the 
comments in \cite[A027907 or A008287]{OEIS}. For the sake of completeness, we
give the following easy argument. By expanding the right-hand side of 
\eqref{3.10} with $\lambda_1=\dots=\lambda_\rho$, we see that $Y_n$ is the
number of all ordered partitions of $n$ consisting of exactly $\rho$ summands
between 0 and $\lambda_1$. To avoid the awkward situation of counting 
different placements of 0 as different (ordered) partitions, we add 1 to all
parts of all partitions. This gives the same number of ordered partitions
(i.e., compositions) of $n+\rho$ into exactly $\rho$ parts between 1 and
$\lambda_1+1$, as claimed. 

(b) As binomial expansion of $(1+q)^{\rho}$, this is obvious.
\end{proof}

Using Lemma~\ref{lem:3.9}, one could now adapt Proposition~\ref{prop:3.8} 
accordingly.

\begin{example}\label{ex:3:10}
{\rm With $\rho=4$ and $\lambda_1=\dots=\lambda_4=2$ we have with \eqref{3.10},}
\begin{equation}\label{3.14}
\sum_{n=0}^8 Y_nq^n = \left(1+q+q^2\right)^4
=1+4q+10q^2+16q^3+19q^4+16q^5+10q^6+4q^7+q^8.
\end{equation}
{\rm The finite sequence of coefficients on the right of \eqref{3.14} 
can be found in \cite[A027907]{OEIS} as a row in what is referred to 
as a triangle of trinomial coefficients.
 
If in \eqref{3.14} we take $n=3$, for instance, then according to 
Lemma~\ref{lem:3.9} we 
consider the compositions of $n+\rho=7$ into $\rho=4$ parts that are bounded
by $\lambda_1+1=3$. Given the corresponding bounded {\it partitions} of 7,
namely $3+2+1+1$ and $2+2+2+1$, we see that the first partition leads to 
$2\cdot\binom{4}{2}=12$ compositions and the second one to 4 compositions, with
a total of 16. This is consistent with \eqref{3.14} and $Y_3=16$.}
\end{example}

\subsection{Different parameters $z_{i,j}$}\label{sec:4.2}

So far in this section, we always chose the parameters $z_{i,j}=1$ for all
$1\leq j\leq\rho$ and $1\leq j\leq\lambda_1$. For our next result, we will set
$z_{i,j}=0$ for many pairs $(i,j)$, while others will remain at 1. To be 
exact, given the integers $\rho\geq 1$ and $\ell\geq 1$, we set
$\Lambda=(\ell, 2\ell,\ldots, \rho\ell)$ in \eqref{3.1} or \eqref{3.9}. We then
choose the parameters $z_{i,j}$ in such a way that, for a fixed $b\geq 1$,
\begin{equation}\label{3.15}
\sum_{n=0}^\infty C_b^\Lambda(n)q^n =\prod_{j=0}^\infty\prod_{i=1}^\rho
\left(1+q^{ib^j}+q^{2ib^j}+\dots+q^{\ell ib^j}\right).
\end{equation}
By expanding the right-hand side of \eqref{3.15}, we see that $C_b^\Lambda(n)$
counts the number of $\rho$-color $b$-ary partitions of $n$, where the 
multiplicity of each part with color $i$ is one of $i, 2i,\ldots, \ell i$.

\begin{example}\label{ex:3.11}
{\rm Let $\rho=4$, $\ell=2$, and $b=2$. Then the first few terms of 
\eqref{3.15} are}
\begin{equation}\label{3.16}
\sum_{n=0}^\infty C_2^\Lambda(n)q^n 
=1+q+3q^2+3q^3+9q^4+9q^5+18q^6+18q^7+39q^8+\dots
\end{equation}
{\rm The 4-color binary partitions of, say, $n=5$ with the $i$th color occurring
$i$ or $2i$ times are as follows, where we bracketed repeated parts for better
visibility:}
\begin{align*}
& 4_1+1_1,\quad (2_1+2_1)+1_1,\quad (2_2+2_2)+1_1,\quad 2_1+1_1+(1_2+1_2),\\
& 2_1+(1_3+1_3+1_3),\quad (1_1+1_1)+(1_3+1_3+1_3),\quad (1_2+1_2)+(1_3+1_3+1_3),\\
& 1_1+(1_2+1_2+1_2+1_2),\quad 1_1+(1_4+1_4+1_4+1_4).
\end{align*}
{\rm There are 9 such partitions, which is consistent with the term $9q^5$
in \eqref{3.16}.}
\end{example}

Next, as we have seen in Proposition~\ref{prop:3.4}, the
corresponding coefficients $Y_n$ satisfy
\begin{equation}\label{3.17}
\sum_{n=0}^\lambda Y_nq^n = \prod_{i=1}^\rho
\left(1+q^i+q^{2i}+\dots+q^{\ell i}\right),
\end{equation}
where $\lambda=\ell+2\ell+\dots+\rho\ell=\ell\rho(\rho+1)/2$. Expanding the
product on the right as usual, we see that $Y_n$ is the number of ordinary
partitions with each part bounded by $\rho$ and occurring at most $\ell$ times.

\begin{example}\label{ex:3.12}
{\rm We let again $\rho=4$ and $\ell=2$. Then \eqref{3.17} has the expansion}
\begin{equation}\label{3.18}
\sum_{n=0}^{20} Y_nq^n 
= 1+q+2q^2+2q^3+4q^4+4q^5+5q^6+5q^7+7q^8+6q^9+7q^{10}+\dots+q^{20}.
\end{equation}
{\rm The ordinary partitions of, say, $n=8$ with each part bounded by 4 and 
occurring at most twice, are}
\begin{align*}
& 4+4,\; 4+3+1,\; 4+2+2,\; 4+2+1+1,\\
& 3+3+2,\; 3+3+1+1,\; 3+2+2+1.
\end{align*}
{\rm There are 7 such partitions, which is consistent with the term $7q^8$ in
\eqref{3.18}.}
\end{example}

Combining \eqref{3.15}, \eqref{3.17}, and their interpretations with 
\eqref{3.2}, we can state the following result.

\begin{proposition}\label{prop:3.13}
Let $\rho\geq 1$, $\ell\geq 1$, and $b\geq 2$ be integers. Then
\begin{enumerate}
\item[(a)] the sequence $C_b^{\Lambda}(n)$ of $\rho$-color $b$-ary 
partitions of $n$ (where the multiplicity of each part with color $i$ is one 
of $i, 2i, \ldots, \ell i$) and
\item[(b)] the sequence $Y_n$ of ordinary partitions with each part bounded by
$\rho$ and occurring at most $\ell$ times
\end{enumerate}
are connected by the recurrence relation
\begin{equation}\label{3.19}
C_b^\Lambda(bn+j)=\sum_{k=0}^{\lfloor\lambda/b\rfloor}Y_{bk+j}C_b^\Lambda(n-k)
\qquad (j=0, 1,\ldots b-1),
\end{equation}
where $\lambda=\ell\rho(\rho+1)/2$.
\end{proposition}

\section{Recurrences for unrestricted $b$-ary partitions}\label{sec:5}

In Section~\ref{sec:3} and earlier in \cite{DE10}, we dealt only with 
{\it restricted\/} $b$-ary partitions and their polynomial analogues, and we 
saw that the recurrence relation \eqref{3.2} was crucial for our results. It is
the purpose of this section to show that a version of \eqref{3.2} also holds 
for the relevant sequences in the {\it unrestricted\/} case. We begin with an 
analogue of 
Defnition~\ref{def:3.1}, but to avoid having to deal with polynomials with
an infinite number of variables, we now refer to the $z_{i,j}$ as parameters.

\begin{definition}\label{def:5.1}
Let $b\geq 2$ and $\rho\geq 1$ be integers. For each $i=1, 2, \ldots, \rho$, 
let $(z_{i,1}, z_{i,2}, \ldots)$ be an infinite sequence of parameters 
satisfying $|z_{i,j}|\leq r^j$ for a fixed real $r\geq 1$, and let $Z$ denote 
the collection of all $z_{i,j}$. Then we define the sequence
$C_b^\rho(n;Z)$ by the generating function
\begin{equation}\label{5.1}
\sum_{n=0}^\infty C_b^\rho(n;Z)q^n =\prod_{j=0}^\infty\prod_{i=1}^\rho
\left(1+z_{i,1}q^{b^j}+z_{i,2}q^{2\cdot b^j}+z_{i,3}q^{3\cdot b^j}+\dots\right).
\end{equation}
When $z_{i,j}=1$ for all $1\leq i\leq\rho$ and $j\geq 1$, we set
$C_b^\rho(n):=C_b^\rho(n;Z)$, so that
\begin{equation}\label{5.2}
\sum_{n=0}^\infty C_b^\rho(n)q^n =\prod_{j=0}^\infty
\left(1+q^{b^j}+q^{2\cdot b^j}+q^{3\cdot b^j}+\dots\right)^\rho
=\prod_{j=0}^\infty\frac{1}{\left(1-q^{b^j}\right)^\rho}.
\end{equation}
\end{definition}

In analogy to \eqref{2.1} combined with Definition~\ref{def:2.1}, we can now 
state the following fact, which is easy to verify by expanding the first 
product in \eqref{5.2} and collecting equal powers of $q$.

\begin{lemma}\label{lem:5.2} 
Let $b\geq 2$ and $\rho\geq 1$ be integers. Then for each integer $n\geq 1$, 
the number $C_b^\rho(n)$ counts the $\rho$-color $b$-ary partitions of $n$.
\end{lemma}

\begin{example}\label{ex:5.3}
{\rm Let $b=2$ and $\rho=2$. Then \eqref{5.2} becomes
\begin{equation}\label{5.3}
\sum_{n=0}^\infty C_2^2(n)q^n = 1+2q+5q^2+8q^3+16q^4+24q^5+40q^6+\ldots;
\end{equation}
see also \cite[A171238]{OEIS}. On the other hand, the 2-color binary partitions
of, say, $n=4$ are}

$4_1,\;4_2,\;2_1+2_1,\;2_1+2_2,\;2_2+2_2,\;2_1+1_1+1_1,\;2_1+1_1+1_2,\;
2_1+1_2+1_2,$

$2_2+1_1+1_1,\;2_2+1_1+1_2,\; 2_2+1_2+1_2,\;1_1+1_1+1_1+1_1,\;
1_1+1_1+1_1+1_2,$

$1_1+1_1+1_2+1_2,\; 1_1+1_2+1_2+1_2,\; 1_2+1_2+1_2+1_2.$

\noindent
{\rm We count 16 partitions, which is consistent with the term $16q^4$ in 
\eqref{5.3}, and is three more than the restricted analogues in 
Example~\ref{ex:2.2}.}
\end{example}

We are now ready to state and prove the unrestricted analogue of 
Proposition~\ref{prop:3.2}.

\begin{proposition}\label{prop:5.4}
Let $b\geq 2$ and $\rho\geq 1$ be integers, and let $Z$ and $C_b^\rho(n;Z)$ be 
as in Definition~\ref{def:5.1}. Then $C_b^\rho(0;Z)=1$, and for integers 
$n\geq 0$ we have
\begin{equation}\label{5.4}
C_b^\rho(bn+j;Z)
= \sum_{k=0}^n Y_{bk+j}(Z)C_b^\rho(n-k;Z) \qquad (j=0, 1,\ldots b-1),
\end{equation}
where the sequence of coefficients $Y_n(Z)$ is defined by the generating 
function
\begin{equation}\label{5.5}
\sum_{n=0}^\infty Y_n(Z)q^n = \prod_{i=1}^\rho
\left(1+z_{i,1}q+z_{i,2}q^2+z_{i,3}q^3+\dots\right).
\end{equation}
\end{proposition}

\begin{proof}
By the conditions $|z_{i,j}|\leq r^j$ for all $i$ and $j$, all infinite series 
and products that occur are uniformly convergent for $|q| < R$, where $R$ is 
a fixed real constant satisfying $R<1/r$. Therefore all manipulations below 
are justified. The proof that follows is based on the proof of Theorem 4.3 in 
\cite{DE10}.

Starting with \eqref{5.1} with $n$ replaced by $m$, we separate the term for 
$j=0$ and shift $j$ by 1 in the remaining factors. If we then use \eqref{5.5}, 
we obtain
\begin{align*}
\sum_{m=0}^\infty C_b^\rho(m;Z)q^m 
&=\prod_{i=1}^\rho
\left(1+z_{i,1}q+z_{i,2}q^2+\dots\right)\\
&\qquad \times\prod_{j=0}^\infty\prod_{i=1}^\rho
\left(1+z_{i,1}q^{b^{j+1}}+z_{i,2}q^{2\cdot b^{j+1}}+\dots\right)\\
&=\left(\sum_{n=0}^\infty Y_n(Z)q^n\right)
\prod_{j=0}^\infty\prod_{i=1}^\rho\left(1+z_{i,1}\left(q^b\right)^{b^j}
+z_{i,2}\left(q^b\right)^{2\cdot b^{j+1}}+\dots\right),
\end{align*}
Using \eqref{5.1} again, we now have
\begin{equation}\label{5.6}
\sum_{m=0}^\infty C_b^\rho(m;Z)q^m
=\left(\sum_{n=0}^\infty Y_n(Z)q^n\right)
\left(\sum_{\nu=0}^\infty C_b^\rho(\nu;Z)q^{b\nu}\right).
\end{equation}
Next, we write
\begin{equation}\label{5.7}
\sum_{n=0}^\infty Y_n(Z)q^n
=\sum_{k=0}^\infty\sum_{j=0}^{b-1}Y_{kb+j}(Z)q^{kb+j}
=\sum_{k=0}^\infty q^{kb}\sum_{j=0}^{b-1}Y_{kb+j}(Z)q^j.
\end{equation}
In \eqref{5.6}, we now set $m=bn+j$, $0\leq j\leq b-1$, and substitute 
\eqref{5.7} into \eqref{5.6}. We then obtain
\begin{align}
\sum_{n=0}^\infty\sum_{j=0}^{b-1}C_b^\rho(bn+j;Z)q^{bn+j}
&=\left(\sum_{k=0}^\infty q^{kb}\sum_{j=0}^{b-1}Y_{kb+j}(Z)q^j\right)
\left(\sum_{\nu=0}^\infty C_b^\rho(\nu;Z)q^{b\nu}\right)\label{5.8} \\
&=\sum_{n=0}^\infty\sum_{j=0}^{b-1}
\left(\sum_{k=0}^n Y_{kb+j}(Z)C_b^\rho(n-k;Z)\right)q^{bn+j},\nonumber
\end{align}
where we have used a Cauchy product with $k+\nu=n$. Finally, if we equate 
coefficients of the (unique) powers $q^{bn+j}$, we obtain \eqref{5.4}, as
desired.
\end{proof}

To conclude this section, we note that the identity \eqref{5.6} can be 
formulated as a functional equation as follows.

\begin{corollary}\label{cor:5.5}
For integers $b\geq 2$ and $\rho\geq 1$, the generating function
\[
F_b^\rho(q;Z):=\sum_{m=0}^\infty C_b^\rho(m;Z)q^m
\]
satisfies the functional equation
\begin{equation}\label{5.9}
F_b^\rho(q;Z) =\left(\sum_{n=0}^\infty Y_n(Z)q^n\right)F_b^\rho(q^b;Z).
\end{equation}
\end{corollary}

We return later to some special cases of the equation \eqref{5.9}.

\section{Applications of the recurrence \eqref{5.4}}\label{sec:6}

In this section we consider two different special cases, corresponding to the
subsections in Section~\ref{sec:4}.

\subsection{All parameters $z_{i,j}=1$}\label{sec:6.1}

As we did before, in this case we set $Y_n:=Y_n(Z)$ and 
$C_b^\rho(n):=C_b^\rho(n,Z)$. Then \eqref{5.5} becomes
\begin{equation}\label{6.1}
\sum_{n=0}^\infty Y_nq^n = \left(1+q+q^2+\dots\right)^\rho=\frac{1}{(1-q)^\rho}
=\sum_{n=0}^\infty\binom{n+\rho-1}{\rho-1}q^n,
\end{equation}
where the right-most binomial identity is well known; see, e.g.,
\cite[eq.~(1.3)]{Go}. By expanding the second term in \eqref{6.1}, we get the
following well-known fact about compositions.

\begin{lemma}\label{lem:6.1}
For integers $\rho\geq 1$ and $n\geq 1$,the number $Y_n$ defined in \eqref{6.1}
is the number of compositions of $n$ into exactly $\rho$ parts. Furthermore,
\begin{equation}\label{6.2}
Y_n = \binom{n+\rho-1}{\rho-1}.
\end{equation}
\end{lemma}

Combining this lemma with Lemma~\ref{5.2} and using Proposition~\ref{prop:5.4},
we get the following result.

\begin{proposition}\label{prop:6.2}
Let $\rho\geq 1$ and $b\geq 2$ be integers. Then
\begin{enumerate}
\item[(a)] the numbers $C_b^\rho(n)$ of $\rho$-color $b$-ary partitions 
of $n$, and 
\item[(b)] the numbers $Y_n$ of compositions of $n$ into $\rho$ parts
\end{enumerate}
are connected by the recurrence relation
\begin{equation}\label{6.3}
C_b^\rho(bn+j)=\sum_{k=0}^n Y_{bk+j}C_b^\rho(n-k) \qquad (j=0, 1,\ldots b-1).
\end{equation}
\end{proposition}

We now consider a few special cases.

\begin{corollary}\label{cor:6.3}
Let $b\geq 2$. The sequence $\left(C_b^1(n)\right)_{n\geq 0}$, which gives the
numbers of $b$-ary partitions of $n$, satisfies the recurrence relations
\begin{equation}\label{6.4}
C_b^1(bn+j)=\sum_{k=0}^n C_b^1(n-k) \qquad (j=0, 1,\ldots b-1),
\end{equation}
and consequently
\begin{align}
&C_b^1(bn) = C_b^1(bn+1) = \dots = C_b^1(bn+b-1),\label{6.5}\\
&C_b^1(bn) - C_b^1(bn-1) = C_b^1(n).\label{6.6}
\end{align}
\end{corollary}

\begin{proof}
The first statement follows from Lemma~\ref{lem:5.2}, and \eqref{6.4} follows
from \eqref{6.3} since by \eqref{6.2} we have $Y_n=1$ for all $n\geq 0$.
The identity \eqref{6.5} is a consequence of \eqref{6.4} since the right-hand
side of \eqref{6.4} does not depend on $j$, for $0\leq j\leq b-1$. Finally,
\eqref{6.6} is obtained by subtracting the corresponding cases of \eqref{6.4}
from each other, which leaves only $C_b^1(n)$ on the right.
\end{proof}

The identities \eqref{6.5} and \eqref{6.6} could also be obtained by 
combinatorial arguments, and \eqref{6.4} then follows from \eqref{6.5} and
\eqref{6.6} by telescoping. The case $b=2$ can be found in entry 
\cite[A018819]{OEIS} under ``Formula". 

Using second differences and the fact that for $\rho=2$ we have $Y_n=n+1$, we
get the following identities after some easy manipulations with the recurrence
\eqref{6.3}.

\begin{corollary}\label{cor:6.4}
Let $b\geq 2$ be an integer. Then the sequence 
$\left(C_b^2(n)\right)_{n\geq 0}$ of $2$-color $b$-ary partitions of $n$ 
satisfies the identities
\begin{align}
&C_b^2(bn+1)-2C_b^2(bn)+C_b^2(bn-1)=0,\label{6.7}\\
&C_b^2(bn)-2C_b^2(bn-1)+C_b^2(bn-2)=C_b^2(n),\label{6.8}
\end{align}
valid for all integers $n\geq 1$.
\end{corollary}

For $b=2$, this result is stated without proof in \cite[A171238]{OEIS}.
The sequence $C_3^2(n)$ is mentioned as A309677 in \cite{OEIS}.

We will now show that the pairs of identities \eqref{6.5}, \eqref{6.6} and
\eqref{6.7}, \eqref{6.8} have natural extensions to all $b$ and $\rho$.

\begin{proposition}\label{prop:6.5}
Let $b\geq 2$ and $\rho\geq 1$ be integers. Then the sequence
$\left(C_b^\rho(n)\right)_{n\geq 0}$ of $\rho$-color $b$-ary partitions of $n$
satisfies the identities
\begin{align}
&\sum_{i=0}^\rho(-1)^i\binom{\rho}{i}C_b^\rho(bn-i)=C_b^\rho(n),\label{6.9}\\
&\sum_{i=0}^\rho(-1)^i\binom{\rho}{i}C_b^\rho(bn+r-i)=0\quad (1\leq r\leq b-1),
\label{6.10}
\end{align}
where we assume that $C_b^\rho(m)=0$ for $m<0$.
\end{proposition}

\begin{proof}
Since we have $z_{i,j}=1$ for all $i$ and $j$, we can combine \eqref{6.1} with 
\eqref{5.9} and expand $(1-q)^\rho$, obtaining
\begin{equation}\label{6.11}
\left(\sum_{m=0}^\infty C_b^\rho(m)q^m\right)
\left(\sum_{j=0}^\rho(-1)^j\binom{\rho}{j}q^j\right)
= \sum_{m=0}^\infty C_b^\rho(m)q^{bm}.
\end{equation}
We now take the Cauchy product on the left, setting $n=m+j$. Then \eqref{6.11}
gives
\begin{equation}\label{6.12}
\sum_{n=0}^\infty\left(\sum_{j=0}^\rho(-1)^j\binom{\rho}{j}C_b^\rho(n-j)\right)q^n
= \sum_{m=0}^\infty C_b^\rho(m)q^{bm},
\end{equation}
with the understanding that $C_b^\rho(\ell)=0$ for $\ell< 0$. By equating the
coefficients of equal powers of $q$ in \eqref{6.12}, we get the following two
cases:

(i) When $n=bm$, $m=0, 1, 2,\ldots$, we get
\[
\sum_{j=0}^\rho(-1)^j\binom{\rho}{j}C_b^\rho(bm-j) = C_b^\rho(m),
\]
which is the same as \eqref{6.9}.

(ii) Similarly, when $n=bm+r$, $1\leq r\leq b-1$, the corresponding sum will
be 0, which is the same as \eqref{6.10}.
\end{proof}

\subsection{Different parameters $z_{i,j}$}\label{sec:6.2}

In this subsection we deal with the unbounded versions of the definitions and
results in Subsection~\ref{sec:4.2}.

\begin{definition}\label{def:6.6}
Let $b\geq 2$ and $\rho\geq 1$ be integers. We choose the parameters $z_{i,j}$ 
in Definition~\ref{def:5.1} in such a way that we obtain the generating 
function
\begin{equation}\label{6a.1}
\sum_{n=0}^\infty\overline{C}_b^\rho(n)q^n=\prod_{j=0}^\infty\prod_{i=1}^\rho
\left(1+q^{ib^j}+q^{2ib^j}+q^{3ib^j}+\dots\right).
\end{equation}
\end{definition}

It is clear that we can rewrite \eqref{6a.1} as
\begin{equation}\label{6a.2}
\sum_{n=0}^\infty\overline{C}_b^\rho(n)q^n=\prod_{j=0}^\infty\prod_{i=1}^\rho
\frac{1}{1-q^{ib^j}}.
\end{equation}

We chose the notation $\overline{C}_b^\rho(n):=C_b^\rho(n;Z)$ to distinguish 
this sequence from $C_b^\rho(n)$ defined in Subsection~\ref{sec:6.1}, where 
in this case $Z$ is implicitly chosen on the right of \eqref{6a.1}. 
By expanding the 
right-hand side of \eqref{6a.1} and collecting powers of $q$, we obtain the
following interpretation of $\overline{C}_b^\rho(n)$.

\begin{lemma}\label{lem:6.7}
Given the integers $b\geq 2$, $\rho\geq 1$, and $n\geq 1$, the coefficient
$\overline{C}_b^\rho(n)$ counts the number of $\rho$-color $b$-ary partitions
of $n$ in which the multiplicity of each part with color $i$ is a multiple
of $i$.
\end{lemma}

\begin{example}\label{ex:6.8}
{\rm We consider the unrestricted version of Example~\ref{ex:3.11} by taking
$b=2$ and $\rho=4$. Then the first few terms of \eqref{6a.1} are
\begin{equation}\label{6a.3}
\sum_{n=0}^\infty\overline{C}_2^4(n)q^n
=1+q+3q^2+4q^3+10q^4+12q^5+24q^6+30q^7+\dots
\end{equation}
We take again $n=5$ and find the 4-color binary partitions in which the parts
of color 1 may occur with any multiplicity, those of color 2 only with even
multiplicities, and those of colors 3 and 4 with multiplicities that are 
multiples of 3 and 4, respectively. In addition to the 9 partitions listed in
Example~\ref{ex:3.11}, we then have
\[
2_1+(1_1+1_1+1_1),\quad (1_1+1_1+1_1)+(1_2+1_2),\quad (1_1+1_1+1_1+1_1+1_1).
\]
We see that the total number of 12 such partitions is consistent with the 
right-hand side of \eqref{6a.3}.}
\end{example}

With the aim of applying Proposition~\ref{prop:5.4}, we now let
$\overline{Y}_n:=Y_n(Z)$, where the collection of parameters $Z$ is again as
in Definition~\ref{6.6}, so that
\begin{equation}\label{6a.4}
\sum_{n=0}^\infty\overline{Y}_nq^n = 
\prod_{i=1}^\rho\left(1+q^i+q^{2i}+q^{3i}+\dots\right)
=\prod_{i=1}^\rho\frac{1}{1-q^i}.
\end{equation}
Upon expanding the product in the middle term, we see that $\overline{Y}_n$
is the number of partitions of $n$ with parts not exceeding $\rho$. It is 
also the number of partitions of $n$ with at most $\rho$ parts; see 
Proposition~\ref{prop:2.4} and Example~\ref{ex:2.5}.

Combining this fact about $\overline{Y}_n$ with Lemma~\ref{lem:6.7} and
Proposition~\ref{prop:5.4}, we get the following result.

\begin{proposition}\label{prop:6.9}
Let $\rho\geq 1$ and $b\geq 2$ be integers. Then
\begin{enumerate}
\item[(a)] the numbers $\overline{C}_b^\rho(n)$ of $\rho$-color $b$-ary 
partitions of $n$ in which the multiplicity of each part with color $i$ is a
multiple of $i$ and
\item[(b)] the numbers $\overline{Y}_n$ of ordinary partitions of $n$ with
parts not exceeding $\rho$
\end{enumerate}
are connected by the recurrence relation
\begin{equation}\label{6a.5}
\overline{C}_b^\rho(bn+j)
=\sum_{k=0}^n\overline{Y}_{bk+j}\overline{C}_b^\rho(n-k) \qquad 
(j=0, 1,\ldots b-1).
\end{equation}
\end{proposition}

We now consider two special cases of Proposition~6.9. First, when $\rho=1$,
then $C_b^1(n)$ is just the number of $b$-ary partitions of $n$, and 
$\overline{Y}_{n}=1$ for all $n$. In this case, Proposition~\ref{prop:6.9}
reduces to Corollary~\ref{cor:6.3}.

For the second special case, $\rho=2$, we require the following well-known
fact; see, e.g., \cite[A004526]{OEIS}. For the sake of completeness, we quote
the simple proof.

\begin{lemma}\label{lem:6.10}
The number of partitions of $n$ with at most two parts is
$\lfloor\frac{n+2}{2}\rfloor$.
\end{lemma}

\begin{proof}
We treat the cases $n$ even and odd separately. When $n=2k$, then all the
allowable partitions of $n$ are $2k, (2k-1)+1, (2k-2)+2, \ldots, k+k$, so there
are $k+1$ of them. When $n=2k+1$, then we count the partitions
$(2k+1), 2k+1, (2k-1)+2, \ldots, (k+1)+k$, and we have again $k+1$ partitions.
These two evaluations can be combined as $\lfloor\frac{n+2}{2}\rfloor$.
\end{proof} 

Combining Lemma~\ref{lem:6.10} with Proposition~\ref{prop:6.9}, we now have the
following result.

\begin{corollary}\label{cor:6.11}
Let $b\geq 2$. then the sequence $\overline{C}_b^2(n)$ satisfies the recurrence
relations 
\begin{equation}\label{6a.6}
\overline{C}_b^2(bn+j)
=\sum_{k=0}^n\left\lfloor\frac{bk+j+2}{2}\right\rfloor\overline{C}_b^2(n-k) 
\qquad (j=0, 1,\ldots b-1),
\end{equation}
where $\overline{C}_b^2(n)$ is the number of $2$-color $b$-ary partitions of
$n$, in which there can be any number of parts with color $1$, but only an 
even number of parts with color $2$.

In particular, we have in the binary case,
\begin{equation}\label{6a.7}
\overline{C}_2^2(2n)=\overline{C}_2^2(2n+1)
=\sum_{k=0}^n(k+1)\overline{C}_2^2(n-k).
\end{equation}
\end{corollary}

The identity \eqref{6a.7} is reminiscent of Corollary~\ref{cor:6.3} and the
paragraph preceding Corollary~\ref{cor:6.4}. Indeed, with some straightforward
manipulations involving \eqref{6a.7}, we obtain the following identity.

\begin{corollary}\label{cor:6.12}
For all $n\geq 2$,
\begin{equation}\label{6a.8}
\overline{C}_2^2(2n)-2\overline{C}_2^2(2n-2)+\overline{C}_2^2(2n-4)
=\overline{C}_2^2(n).
\end{equation}
\end{corollary}

By the left equation in \eqref{6a.7} we have 
\begin{equation}\label{6a.8a}
\overline{C}_2^2(2n+1)-\overline{C}_2^2(2n)=0\quad (n\geq 0);
\end{equation}
this can be complemented by the following identity, which is easy to obtain 
from the right equation in \eqref{6a.7}.

\begin{corollary}\label{cor:6.12a}
For all $n\geq 1$, we have
\begin{equation}\label{6a.8b}
\overline{C}_2^2(2n)-\overline{C}_2^2(2n-1)
=\sum_{k=0}^n\overline{C}_2^2(k).
\end{equation}
\end{corollary}

We illustrate the identities \eqref{6a.8}, \eqref{6a.8a}, and \eqref{6a.8b}
with an example.

\begin{example}\label{ex:6.12}
{\rm With $b=\rho=2$ in \eqref{6a.2}, we compute
\begin{gather}
\sum_{n=0}^\infty\overline{C}_2^2(n)q^n
=\prod_{j=0}^\infty\frac{1}{\left(1-q^{2^j}\right)\left(1-q^{2\cdot 2^j}\right)}\label{6a.9} \\
=1+q+3q^2+3q^3+8q^4+8q^5+16q^6+16q^7+32q^8+32q^9+56q^{10}+\dots \nonumber
\end{gather}
In addition to the obvious illustration of \eqref{6a.8a}, we see that for 
$n=5$ in \eqref{6a.8} and \eqref{6a.8b} we have $56-2\cdot 32+16=8$ and
$56-32=1+1+3+3+8+8$ respectively, as expected.}
\end{example}

In analogy to Proposition~\ref{prop:6.5}, we can vastly extend 
Corollary~\ref{cor:6.12}. To do so, we introduce the following coefficients
related to the $q$-Pochhammer symbol $(q;q)_n$, using the notation from
\cite[A231599]{OEIS}; see also \cite[A333290]{OEIS}.

\begin{definition}\label{def:6.14}
{\rm For integers $n\geq 0$ and $0\leq j\leq n(n+1)/2$, we define the
coefficients $T(n,j)$ by
\begin{equation}\label{6a.10}
(q;q)_n = \prod_{i=1}^n\left(1-q^i\right)
= \sum_{k=0}^{\frac{n(n+1)}{2}}T(n,k)q^k,
\end{equation}
with the convention that $T(0,0)=1$.}
\end{definition}

The first few coefficients $T(n,k)$ are listed in Table~1.

\bigskip
\begin{center}
{\renewcommand{\arraystretch}{1.1}
\begin{tabular}{|l|c c c c c c c c c c c|}
\hline
$n \backslash k$ & 0 & 1 & 2 & 3 & 4 & 5 & 6 & 7 & 8 & 9 & 10 \\
\hline
0 & 1 & & & & & & & & & &\\
1 & 1 & $-1$ & & & & & & & & &\\
2 & 1 & $-1$ & $-1$ & 1 & & & & & & & \\
3 & 1 & $-1$ & $-1$ & 0 & 1 & 1 & $-1$ & & & & \\
4 & 1 & $-1$ & $-1$ & 0 & 0 & 2 & 0 & 0 & $-1$ & $-1$ & 1 \\
\hline
\end{tabular}}

\medskip
{\bf Table~1}: $T(n,k)$ for $0\leq n\leq 4$.
\end{center}

\bigskip
The coefficients $T(n,k)$ have earlier been used in \cite[Theorem~1.3]{MS}, 
with a different notation. Some recurrence relations for the coefficients 
$T(n,k)$ are also given in \cite{MS}. Finally,
it is worth mentioning that \eqref{6a.10} turns into Euler's pentagonal number 
theorem if we take the limit as $n\to\infty$; see, e.g., \cite[pp.~311 ff.]{Ap}.
We can now state and prove the following result. 

\begin{proposition}\label{prop:6.15}
Let $b\geq 2$ and $\rho\geq 1$ be integers. Then the sequence
$\left(\overline{C}_b^\rho(n)\right)_{n\geq 0}$ of $\rho$-color $b$-ary 
partitions of $n$ in which the multiplicity of each part with color $i$ is a 
muliple of $i$, satisfies the identities
\begin{align}
&\sum_{j=0}^{\frac{\rho(\rho+1)}{2}}T(\rho,j)\overline{C}_b^\rho(bn-j)
=\overline{C}_b^\rho(n),\label{6a.11}\\
&\sum_{j=0}^{\frac{\rho(\rho+1)}{2}}T(\rho,j)\overline{C}_b^\rho(bn+r-j)
= 0 \qquad (1\leq r\leq b-1),\label{6a.12}
\end{align}
where we assume that $\overline{C}_b^\rho(m)=0$ for $m<0$.
\end{proposition}

\begin{proof}
We proceed as in the proof of Proposition~\ref{prop:6.5}. We consider 
\eqref{5.9} with $Z$ as in Section~\ref{sec:4.2}. Using \eqref{6a.4}, we then
obtain
\[
\left(\sum_{m=0}^\infty\overline{C}_b^\rho(m)q^m\right)
\prod_{i=1}^\rho\left(1-q^i\right)
= \sum_{m=0}^\infty\overline{C}_b^\rho(m)q^{bm}
\]
and with \eqref{6a.10},
\begin{equation}\label{6a.13}
\left(\sum_{m=0}^\infty\overline{C}_b^\rho(m)q^m\right)
\left(\sum_{j=0}^{\frac{\rho(\rho+1)}{2}}T(\rho,j)q^j\right)
= \sum_{m=0}^\infty\overline{C}_b^\rho(m)q^{bm}.
\end{equation}
Taking the Cauchy product on the left with $n=m+j$, we obtain from 
\eqref{6a.13},
\begin{equation}\label{6a.14}
\sum_{n=0}^\infty\left(\sum_{j=0}^{\frac{\rho(\rho+1)}{2}}T(\rho,j)
\overline{C}_b^\rho(n-j)\right)q^n
= \sum_{m=0}^\infty\overline{C}_b^\rho(m)q^{bm},
\end{equation}
with the understanding that $\overline{C}_b^\rho(\ell)=0$ for $\ell< 0$. 
Finally, by distinguishing between the two cases $n=bm$ and $n=bm+r$,
$1\leq r\leq b-1$, \eqref{6a.14} gives the two identities \eqref{6a.11}
and \eqref{6a.12}, respectively.
\end{proof}

As a first special case, we consider $\rho=1$. Then 
$\overline{C}_b^\rho(n)=C_b^1(n)$, the number of $b$-ary partitions of $n$.
With $T(1,0)=1$ and $T(1,1)=-1$, the identities \eqref{6a.11} and \eqref{6a.12}
reduce to \eqref{6.6} and \eqref{6.5}, respectively.

The case $\rho=2$ is more interesting and we present it as a corollary. 
We recall that for an integer $b\geq 2$, $\overline{C}_b^2(n)$ is the number of
2-color $b$-ary partitions of $n$ in which any number of parts can have color 1,
but only an even number of parts can have color 2.

\begin{corollary}\label{cor:6.16}
For any integers $b\geq 2$ and $n\geq 1$, we have
\begin{equation}\label{6a.15}
\overline{C}_b^2(bn)-2\overline{C}_b^2(bn-2)+\overline{C}_b^2(bn-4)
=\overline{C}_b^2(n),
\end{equation}
and for $b\geq 3$,
\begin{equation}\label{6a.16}
\overline{C}_b^2(bn-1)-2\overline{C}_b^2(bn-3)+\overline{C}_b^2(bn-5) = 0.
\end{equation}
In both identities we assume that $\overline{C}_b^2(m)=0$ when $m<0$.
\end{corollary}

\begin{proof}
We set $\rho=2$ in \eqref{6a.11}, obtaining
\begin{equation}\label{6a.17}
\overline{C}_b^2(bn)-\overline{C}_b^2(bn-1)-\overline{C}_b^2(bn-2)
+\overline{C}_b^2(bn-3)=\overline{C}_b^2(n),
\end{equation}
where we have used the relevant entry in Table~2. Similarly, we set $\rho=2$
and $r=b-1$ in \eqref{6a.12} and replace $n$ by $n-1$, getting
\begin{equation}\label{6a.18}
\overline{C}_b^2(bn-1)-\overline{C}_b^2(bn-2)-\overline{C}_b^2(bn-3)
+\overline{C}_b^2(bn-4)=0.
\end{equation}
By adding \eqref{6a.17} and \eqref{6a.18}, we immediately get \eqref{6a.15}. 

Finally, when $b\geq 3$, we can also take $r=b-2$ in \eqref{6a.12}. This gives
us an analogue of \eqref{6a.18}, with the arguments shifted by 1. Adding this
to \eqref{6a.18} then leads to \eqref{6a.16}.
\end{proof}

It is clear from the proof of \eqref{6a.16} that for larger $b\geq 3$, more
such 3-term recurrences can be obtained. This is related to the fact that 
\eqref{6a.12} is actually a collection of $b-1$ identities. On the other hand,
Table~1 indicates that we cannot expect simplified identies as in 
Corollary~\ref{cor:6.16} for $\rho\geq 3$.

We conclude this section with an example.

\begin{example}\label{ex:6.17}
{\rm With $b=3$ and $\rho=2$ in \eqref{6a.2}, we compute
\begin{gather}
\sum_{n=0}^\infty\overline{C}_3^2(n)q^n
=\prod_{j=0}^\infty\frac{1}{\left(1-q^{3^j}\right)\left(1-q^{2\cdot 3^j}\right)}\label{6a.19} \\
=1+q+2q^2+3q^3+4q^4+5q^5+8q^6+9q^7+12q^8+16q^9+19q^{10}+23q^{11}+\dots \nonumber
\end{gather}
With $n=3$ in \eqref{6a.15}, we have $16-2\cdot 9+5=3$, as expected, and with 
$n=4$ in \eqref{6a.16}, we have $23-2\cdot 16+9=0$, again as expected.

To illustrate the combinatorial interpretation of $\overline{C}_3^2(n)$, we
choose $n=7$ and begin with the (uncolored) base-3 partitions $3+3+1$, 
$3+1+1+1+1$, and $1+\dots+1$. These lead to the following colored partitions,
keeping in mind that only an even number of parts can have color 2:
\begin{align*}
&3_1+3_1+1_1,\quad 3_2+3_2+1_1;\\
&3_1+1_1+1_1+1_1+1_1,\quad 3_1+1_1+1_1+1_2+1_2,\quad 3_1+1_2+1_2+1_2+1_2,
\end{align*}
and four partitions, each with seven parts ``1" that have either no color 2, or
have two, four or six parts with color 2. The total number of such colored
partitions is therefore 9, which is consistent with the term $9q^7$ on the
right of \eqref{6a.19}.}
\end{example}

\end{document}